\RequirePackage{fix-cm}
\documentclass[smallextended]{svjour3}       

\smartqed  
\usepackage{graphicx}
\usepackage{amssymb}
\usepackage{amsmath}

\newcommand\inv[1]{#1\raisebox{1.15ex}{$\scriptscriptstyle-\!1$}}
\usepackage{tikz}
\usetikzlibrary{shapes,arrows,calc}

\newtheorem{coro}{Corollary}
\newtheorem{defn}{Definition}
\newtheorem{ex}{Example}

\begin{document}
\title{On maximum distance separable group codes with complementary duals}


\author{Saikat Roy         \and
        Satya Bagchi 
}


\institute{S. Roy \at
                Junior Research Fellow, CSIR\\
                Department of Mathematics\\ 
                National Institute of Technology Durgapur\\
                Burdwan, India. \\
               \email{saikatroy.cu@gmail.com}           
           \and
           S. Bagchi \at
                Department of Mathematics\\ 
                National Institute of Technology Durgapur\\
                Burdwan, India. \\
               \email{satya.bagchi@maths.nitdgp.ac.in}           
}

\date{Received: date / Accepted: date}

\maketitle

\begin{abstract}
Given an LCD group code $C$ in a group algebra $KG$, we inspect kinship between $C$ and $G$, more precisely between the subgroup structures of $G$ and $C$. Under some special circumstances our inspection provides an estimation for various parameters of a group code $C$. When $C$ is MDS, the inter relation between $K$ and $G$ becomes more impressive. Application of Sylow theorem facilitated us to explore the inter relation between $G$ and $K$ (when $C$ is LCD and MDS) in more general way and finally we get the result of Cruz and Willems (Lemma $4.4$) as a simple sequel.

\end{abstract}

\section{Introduction}
A group algebra \cite{DummitFoote} \cite{MacWilliamsSloane} composed of two algebraic objects, a field and a finite group. For $K$ being a field and $G$ being a finite group, the group algebra is denoted by $KG$ and defined by the collection of formal sums $$\lbrace \sum_{g\in G}a_gg:a_g\in K \rbrace.$$

Under the binary operations ``$+$'' and ``$\cdot$'' defined by $\forall a, b\in KG$, $a=\sum_{g\in G}a_gg$, $b=\sum_{g\in G}b_gg$; $a_g, b_g\in K$; $$ a+b=\sum_{g\in G}(a_g+b_g)g ~~\mathrm{and} ~~a.b=\sum_{g\in G}(\sum a_gb_g)g$$ respectively, with $1.g=g,~ 0.g=0,~\forall g\in G$, $(1$ and $0$ being multiplicative identity and additive identity in the field $K$ respectively) and $a.id=a,~\forall a\in K$, where $id$ is the identity element in the group $G$,
$(KG, +, \cdot)$ forms a vector space over $K$ (with the elements of $G$ as a basis) as well as a ring containing $K$ and $G$. $KG$ becomes commutative if $G$ is commutative. Since $K$ commutes with every member of $KG$ so $K \subseteq C(KG)$, where $C(KG)$ is the centre of $KG$.


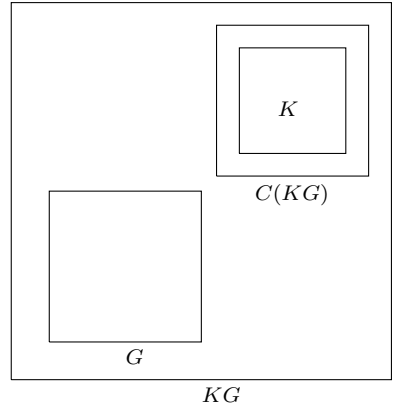
\begin{figure}[ht]\label{p1_f1}
\begin{center}
\begin{tikzpicture}

\draw (0,0) -- (5,0) -- (5,5) -- (0,5) -- (0,0);
\draw (0.5,0.5) -- (2.5,0.5) -- (2.5,2.5) -- (0.5,2.5) -- (0.5,0.5);
\draw (2.7,2.7) -- (4.7,2.7) -- (4.7,4.7) -- (2.7,4.7) -- (2.7,2.7);
\draw (3,3) -- (4.4,3) -- (4.4,4.4) -- (3,4.4) -- (3,3);
\draw (2.4,0)   node[anchor=north west] {$KG$};
\draw (1.4,0.5) node[anchor=north west] {$G$};
\draw (3.1,2.7) node[anchor=north west] {$C(KG)$};
\draw (3.4,3.8) node[anchor=north west] {$K$};
\end{tikzpicture}

\caption{Group algebra $KG$}
\end{center}
\end{figure}

Since $K \subseteq C(KG)$, $KG$ sometimes called $k$-algebra. The structure of the group algebra $KG$ can be visualized (intuitively) from Figure \ref{p1_f1}. The dimension of the vector space $KG$ over $K$ is the cardinality of $G$.   A group code $C$ is a right ideal of $KG$ with the length $= |G|$ and weight function $wt(a)= |\{g\in G ~| a_g \neq 0\}|$.

Since $KG$ is a vector space over $K$, there is no harm to talk about the inner product in $KG$. The inner product in $KG$ is defined in a similar way the dot product is defined in Euclidean space. The inner product of $a, b \in KG$, denoted as $\langle a,b\rangle$, is defined by $\langle a,b\rangle =\sum a_gb_g.$ The basic difference between Euclidean dot product and the inner product in $KG$ is, $||a||$ can become zero without ``$a$'' being zero in $KG$. $||a||=0$ without $a$ being zero in $KG$ happens iff characteristic of $K$ divides $\sum a_g^2$.  Also for any $g\in G$, $\langle ag,bg\rangle =\langle a,b\rangle$ and hence this inner product is called $G$-invariant. Using this we can characterize the dual of $C$ as $`` b\in C^\perp $ iff $\langle a,b\rangle =0,~ \forall a\in C."$ 

A group code $C$ is called linear complementary dual (LCD) \cite{Massey} iff $C\cap C^\perp=\lbrace 0 \rbrace.$ LCD group codes are asymptotically good, satisfies Gilber-Varshamov bound and has significant applications in modernistic information theory \cite{CruzWolfgang}. Mathematically the LCD group codes are more sophisticated than just ordinary linear codes. Where an ordinary linear code involves the concept of vector spaces only, the group code involves the concept of rings, fields, ideals, modules and vector spaces.

In this paper, the ideals we are considering are right ideals, if nothing else mentioned. By $KG$, we mean the $k$-algebra, where $G$ is a group of order $n$ and $K$ is a finite field. We denote idempotents in $KG$ as $z$. By ``*" denotes the binary operation in the group.

\begin{defn}
Let $a\in KG$ with $a=\sum_{g\in G}a_gg,a_g\in K$, the adjoint of $a$, $\hat a$ is defined by $\sum_{g\in G}a_g\inv{g}$.
\end{defn}

\begin{defn}
A group code $C$ is called self adjoint if $c\in C$ implies $\hat c\in C$.
\end{defn}

Instead of a field, if $K$ happens to be a ring, $KG$ mere forms a module over $K$. A module called projective iff it is a direct summand of a free module.

\begin{defn}
$z\in KG$ is called an idempotent if $z^2=z$, and if $z$ belongs to the centre of $KG$ then $z$ is called a central idempotent.
\end{defn}

\begin{defn}
A $q$-ary linear code $C$ having length $n$, dimension $k$, minimum distance $d$, symbolically $[n,k,d]_q,$ is called maximum distance separable (MDS) code, if $d=n-k+1$. For a particular code $C$, we denote this minimum distance of $C$ by $d(C)$.
\end{defn}

If $C$ is an LCD group code in a group algebra $KG$ then $C$ and $C^{\perp}$ can be expressed as $zKG$ and $(1-z)KG$ respectively, where $z=z^2=\hat{z} \in KG$. In addition the LCD group code becomes self adjoint iff $z\in C(KG)$. Algebraic structure that an LCD group code $C$ carries is very much dependent on the structures of underlying field and group of the group algebra.

This motivates us to study the properties of the underlying group subject to a group code. While the result given by Cruz and Willems in \cite{CruzWolfgang} is very much decisive for the commutative groups it neither affirms nor negates the same conclusion for non-commutative groups. Application of Sylow theorems ensures something more regarding an LCD, MDS group code.

Given $G$ be any group having order $p^nm$, with $gcd(p,m)=1$, and $p$ being a prime. Then the Sylow's theorems \cite{DummitFoote} conclude the following. 
\begin{enumerate}
\item There exists a subgroup of $G$ having order $p^n$ called Sylow-$p$-subgroup of $G$.
\item Any two Sylow-$p$-subgroup of $G$ are conjugate.
\item Number of Sylow-$p$-subgroups of $G$ is $n_p\equiv 1\pmod p$ and $n_p\mid |G|$.
\end{enumerate}

Furthermore, it is easy to see that the unique Sylow subgroups are normal and converse. In addition, for a commutative group $G$, all Sylow subgroups are normal, hence unique.

In \cite{Willems} Willems completely classified the group algebras which contains a self dual ideal, also in \cite{CruzWolfgang} Willems established that ``For $G$ being a finite abelian group and $K$ being a field of characteristic $p$ with $KG$ contains an LCD group code $C$. If $C$ is MDS, then $p\nmid |G|$." Here we have proved the same result for a larger class of groups. As we go a bit further, we see that Sylvester's Rank-Nullity theorem \cite{HoffmanKunze} provides us an estimation about the distance of an LCD group code. For $T$ being a linear operator from a finite dimensional vector space $V$ (over some field $K$), Sylvester's Rank-Nullity theorem ensures the following fact
$dimV=dim(\mathrm{Range~of} ~T)+dim(\mathrm{Kernel~of}~T)$, i.e., $$dimV=rank(T)+nullity(T).$$

The remaining portion of this paper consists of main results, given in Section $2$ and some relevant examples (in context of our results) in Section $3$. We close our paper by drawing some conclusions in Section $4$.


\section{Results} 
Throughout this section the cardinality of a set $X$ is denoted as $|X|$ and $dim(KG)$ denotes the dimension of the vector space $KG$ over the field $K$.  By ``the coefficients of any element $a\in KG$" we mean $a_g\in K$ and by the components of $a$ we mean different $g\in G$ associated to the different non zero $a_g$'s.

\textbf{Alert!!!}
 In next two theorems we should note that the statements ``components of $a\in KG$ with identity \textbf{forms} a subgroup of $G$" and ``components of $a\in KG$ \textbf{generates} a subgroup of $G$" are not the same. While ``the components of $a\in KG$ with identity forms a subgroup of $G$" means when we consider the components of $a$ as a set together with identity, the set remains closed and associative under the binary operation of $G$. The set is closed under inverse, i.e., the set together with identity qualifies to be a subgroup of $G$, ``the components of $a\in KG$ generate subgroup of $G$" means the subgroup in $G$ generated by the components of $a$. As for example let $G=S_3$ and $K$ be any field, if we set $a=(12)$ then component of $a$ with unity is already a subgroup of $S_3$ but if we set $a=()+(123)$, the components of $a$ is not a subgroup of $S_3$, and the subgroup generated by the components of $a$ is $\langle (123) \rangle$.

\begin{theorem}\label{p1_t1}
Suppose $z\in KG$ is an element such that the non-zero coefficients of $z$ is $\lambda$ for some fixed $\lambda \in K$, if components of $z$ with identity (we include identity separately if it do not appears in the components of $z$) form a normal subgroup of $G$ . Then $z\in C(KG)$.
\end{theorem}

\begin{proof} 
Suppose $z=\lambda (h_1+h_2+\cdots +h_k) \in KG$. Given the components of $z$ along with identity (we include identity separately if it do not appears in the components of $z$) already qualifies to be a subgroup of $G$, call it $Z$. Now for each $g\in G$, we have $gZg^{-1}=Z$, since $Z\trianglelefteq G$ by our hypothesis. Now $g\cdot z=g\cdot \lambda (h_1+h_2+\cdots +h_k)=\lambda (h_1+h_2+\cdots +h_k)\cdot g,\forall g\in G,$ since $g\cdot h_i=h_j\cdot g$ for $1\leq i,j\leq k$ (for elements of $KG$ which comes from $G$ only ``$\cdot$'' and ``$\ast$'' coincides). Therefore $z$ commutes with every element of the group $G$, since the non zero coefficients of $z$ is $\lambda$, therefore it commutes with every element of $KG$. Hence $z\in C(KG)$, as desired.
\end{proof}

\begin{coro}
Because the inversion is a bijective map we get $z=\hat{z}$. In addition if such $z$ is idempotent in $KG$, we can construct self-adjoint, LCD group code $zKG$.
\end{coro}

\begin{theorem}
Let $C=zKG$, where $z=z^2=\hat z$ be a self adjoint group code in $KG$, then the components of $z$ generate a normal subgroup in $G$.  
\end{theorem}
\begin{proof}
We know that for $zKG$ being a self adjoint LCD group code, $z\in C(KG)$ and hence $z$ commutes with every member of $G$. Now for any member $g\in G$, we have $z\cdot g=g\cdot z$. Considering $G=\lbrace  g_1,g_2,\dots ,g_n\rbrace $ and components of $z$ as $\lbrace h_1, h_2, \cdots, h_k\rbrace $. Let $z=\sum_{i=1}^kz_ih_i,z_i\in K$. Since $z\cdot g=g\cdot z, \forall g\in G$, therefore, $z_i(g*h_i)=z_j(h_j*g),1\leq i,j\leq k$, i.e., $(z_i\inv{z_j})(g*h_i)=(h_j*g)$. Since $KG$ is a vector space over $K$, having the elements of $G$ as a basis, hence if $(g*h_i)$ and $(h_j*g)$ are distinct in $KG$ then $(g*h_i)$ and $(h_j*g)$ are linearly independent. Therefore $(z_i\inv{z_j})(g*h_i)=(h_j*g)$ holds if and only if $(z_i\inv{z_j})=1$ and $h_j*g=g*h_i$.

Now we consider $Z=\langle h_1,h_2,\dots ,h_k\rangle$, then $g*(h_{i_1}*\dots *h_{i_k}) = (h_{i_1}'*\dots *h_{i_k}')*g$ where $h_{i_1},\dots ,h_{i_k},h_{i_1}',\dots ,h_{i_k}'$ are among the components of $z$. Therefore $gZ=Zg,~\forall g\in G$, proving $Z\trianglelefteq G$. 
\end{proof}

Our next theorem ensures the existence of a self-adjoint LCD group code if the mother group $G$ has a normal subgroup of a particular order.

\begin{theorem}
Let $KG$ be a $k$-algebra with characteristic of $K$ is $p$, then for each subgroup $H\in G$ with $|H|\equiv 1\pmod p$, there exists an LCD group code in $KG$. Furthermore, if $H\trianglelefteq G$, then the group code becomes self-adjoint.
\end{theorem}

\begin{proof}
Let $H$ be a subgroup in $G$. Put $z=\sum_{h\in H}h$, then it is obvious that $z=\hat z$. To show $z$ is idempotent we formally calculate $z^{2}$, now $z^{2}=|H|\sum_{h\in H}h=(pk+1)\sum _{h\in H}h=\sum_{h\in H}h$, for some positive integer $k$.

If $H\trianglelefteq G$, then we set $z$ as above, then by Theorem \ref{p1_t1}, $z \in C(KG)$ and therefore $zKG$ is self-adjoint group code.
\end{proof}

\begin{coro} Let $G$ be a finite group and $K$ be a field with characteristic $p$. If $H\leq G$ such that $H$ is a characteristic subgroup of $G$ or unique Sylow subgroup of $G$ or $[G:H]=q$, where $q$ is the lowest prime divide the order of the group $G$ and $|H|=pk+1$ for some $k\in \mathbb{N}$. We can construct an LCD self-adjoint group code with the help of $H$.
\end{coro} 
\begin{proof}
If $H$ satisfies any of three conditions stated above $H$ qualifies to be a normal subgroup of $G$ having order $pk+1$ for some $k\in \mathbb{N}.$ Then proceeding in exactly same manner described in Theorem 3, we can form a self-adjoint LCD group code in $KG$.
\end{proof}

\begin{theorem}
Let $(G,*)$ be a group of order $p^n m$, with $gcd(p,m)=1$ and $K$ be a field of characteristic $p$.
\begin{enumerate}
\item For $m=1$, $KG$ do not contains a proper LCD code.
\item Let $C=zKG$, where $z=z^2=\hat z$, be an LCD group code then the components of $z$ is a subset of $\lbrace\cup$\{Sylow-$p$-subgroups\}$\setminus\lbrace id\rbrace  \rbrace^c$, complement of $\lbrace\cup$\{Sylow-$p$-subgroups\}$\setminus\lbrace id\rbrace  \rbrace$.
\end{enumerate}
\end{theorem}

\begin{proof}
1. Suppose $KG$ contains a proper LCD code $C$, then $C=zKG$, where $z=z^2=\hat z$, If identity is the only component of $z$ then $C=KG$, if there are some non trivial components in $z$ then $z^{p^n}= \lambda \cdot id$, for some scaler $\lambda \in K$; but according to our hypothesis $z=z^2=\cdots =z^{p^n}=z$, hence contradiction, so no such code exists.\\

2. For any $g\in$ Sylow-p-subgroup $P$ of $G$, $g^{p^n}=id$. So if some $g\in P$ appears in the components of $z$, we get $z^{p^n} \neq z$. So contradiction. Hence proved.
\end{proof}

\begin{theorem}
Let $G$ be a finite group and $K$ is a field of characteristic $p$ and $p\mid |G|.$ Suppose $KG$ contains an LCD and MDS group code $C$. Then $d(C)\geq p^m+1$, where $p^m$ is the order of Sylow-p-subgroup in $G$.
\end{theorem}

\begin{proof}
Since $C$ is an LCD and MDS code in $KG$ so $C$ is of the form $zKG$, where $z^2=\hat z=z$; Naturally, (by second part of Theorem $4$) $wt(z)\leq n-p^m+1$ and $wt(1-z)\leq n-p^m+1$. Now $d(C^{\perp})\leq n-p^m+1$.

$$d(C^{\perp})=n-dim(C^{\perp})+1 = dim(C)+1$$ $$ \therefore dim(zKG)+1\leq n-p^m+1 \Rightarrow dim(zKG)\leq n-p^m.$$ Since $dim(zKG)+dim((1-z)KG)=n,$ therefore, $dim((1-z)KG)\geq p^m.$ Because $d(C)=dim((1-z)KG)+1$, we get $d(C)\geq p^m+1.$ Similarly, $d(C^{\perp})\geq p^m+1.$
\end{proof}

\begin{coro}
The bounds of dimension $zKG$ and minimum distance can be improved once the number of Sylow-p-subgroup is determined. Since, the more Sylow-p-subgroup of $G$ implies the more elements in $G$  having order $p^r$ (for some integer $r$).
\end{coro}

\begin{theorem}
Suppose $G$ is a finite group and $K$ is a field of characteristic $p$. Let $C=zKG,$ where $z=z^2=\hat z$; be an LCD, MDS group code in $KG$. Define $S=\{g\in G: z\cdot g=scaler~multiple~of~z\}$.
 
1. $S$ is a subgroup of $G$. 

2. $d(C)\geq |S|$. 
\end{theorem}

\begin{proof}
1. Since the identity element of $G$ lies in $S$, $S\neq \phi$. Consider $g,h\in S$, therefore $z g=\lambda z$, and $z h=\mu z $ for some $\lambda ,\mu \in K.$ Naturally $\mu ^{-1}z=z h^{-1}.$

Now $zgh^{-1}=\lambda  z h^{-1}=\lambda \mu ^{-1} z$, and $\lambda \mu^{-1}\in K.$ So $\forall g,h\in S,~gh^{-1}\in S,$ proving $S$ is a subgroup of $G$.

2. Any element in $zKG$ is of the form $z\cdot a=\sum_{g\in G}a_gz\cdot g $, i.e., $$z\cdot a=\sum_{g\in S}a_gz\cdot g + \sum_{g\not\in S}a_gz\cdot g=\lambda z+\sum_{g\not\in S}a_gz\cdot g $$
for some $\lambda \in K.$ Therefore $$dim(zKG)\leq n-|S|+1\Rightarrow n-d(C)+1\leq n-|S|+1\Rightarrow d(C)\geq |S|.$$
Hence proved.
\end{proof}

\textbf{Remark} The lower bound of $d(C)$, that has been obtained in Theorem $5$ and $6$ is useful, although we have already assumed that our group code is LCD, MDS. The main advantage of having these bounds is that we can still give some estimation of $d(C)$ without knowing the dimension of $C$.

\begin{coro}
For $S_1=\{g\in G:(1-z)\cdot g=scaler~multiple~of~(1-z)\}$. $S_1$ is a subgroup of $G$ and $d(C^{\perp})\geq |S_1|.$
\end{coro}

\begin{theorem}
If characteristic of $K(=p)\mid |G|$ and number of elements in $G$ having order $p^r$ (for some integer $r$) is greater than $\frac{|G|}{2}-1$, then $KG$ cannot contain an LCD, MDS group code.
\end{theorem}

\begin{proof}
Let $C=zKG$, $z=z^2=\hat{z}$, be an LCD, MDS group code contained in $KG$. Let $D=\{ g\in G : |g|=p^r\}\subset G$ (for some integer $r$). We have $d(C^{\perp})\leq n-|D|$, since $(1-z)\in C^{\perp}$ and any component of $z$ cannot be an element of order $p^r$ in $G$. 

By assumption $C$ is MDS so $d(C^{\perp})=|G|-dim((1-z)KG)+1$, therefore $dim(zKG)+1\leq n-|D|$. Similarly $dim((1-z)KG)+1\leq n-|D|$; adding duo we get $n+2\leq 2n-2|D|,$ finally $|D|\leq \frac{|G|}{2}-1.$
\end{proof}

\begin{theorem}
$dim(KG)=dim(zKG)+|$linear independent (LI) right zero divisors of $z|$, where $z=z^2 = \hat{z}$.
\end{theorem}

\begin{proof} Define $\varphi: KG\rightarrow zKG$ by $\varphi(a)=z\cdot a$, where ``$\cdot$'' is the usual multiplication on $KG$.

Claim: $\varphi$ is a linear map. Now, for $a, b \in KG$ and $\alpha, \beta \in K$, we get
\begin{eqnarray*}
\varphi (\alpha a+\beta b) &=& z\cdot(\alpha a+\beta b), \\
                            &=&  z\cdot(\alpha a) + z\cdot(\beta b),~\mathrm{using ~distributive ~law},\\
                            &=&  (z\cdot\alpha)(a)+(z\cdot\beta) (b),~\mathrm{using ~associativity},\\
                            &=& (\alpha \cdot z)(a)+(\beta \cdot z)(b),~\mathrm{since}~K\in C(KG),\\
                            &=& \alpha (z\cdot a)+\beta (z \cdot b),~\mathrm{using~associativity},\\
                            &=& \alpha \varphi(a)+\beta \varphi(b).
\end{eqnarray*}\\
                            Kernel of $\varphi$ is precisely the set of all right zero divisors of $z$ and dimension of kernel will be the number of linearly independent right zero divisors of $z$. By the construction the mapping is onto and hence by Sylvester's Rank-Nullity theorem we conclude the result.
\end{proof}                       
 
\begin{coro} Proceeding in exactly similar way and putting $(1-z)$ instead of $z$ we get\\ \[dim(KG) = dim((1-z)KG)+|LI~right~zero~dividors~of~(1-z)|.\]
\end{coro}                            
                            
\begin{coro} For an $[n,k,d]_q$ singleton bound permits us to write $d \leq n-k+1$, so by the above theorem if KG contains an LCD group code $C$ of dimension $k$ with $|G|=n$ then $n-k=|$LI right zero divisors of $z|$, and hence $C$ is given by $d(C)\leq|$LI right zero divisors of $z|+1$. For an MDS linear $[n,k,d]_q$ code $d(C)=|$LI right zero divisors of $z|$+1.
\end{coro}

\begin{theorem}
Let $G$ be a finite group of order $n$ and $K$ be a field of characteristic $p$ and suppose $C=zKG$, where $z=z^2=\hat z$ is an LCD group code. If $C$ is an MDS code then the components of $z$ generate $G$. 
\end{theorem}

\begin{proof} Suppose $z$ has components ${h_1,h_2, \dots, h_t}$. Now let $H$ be the subgroup of $G$ generated by  ${h_1,h_2, \dots, h_t}$. We have, $$dim(KG)-dim(zKG)+1=d(C).$$ $$\mathrm{Also},~~dim(KG)-dim((1-z)KG)+1=d(C^\perp).$$
                            
Adding these two equalities, we get $$2\cdot dim(KG)-dim(zKG)+1-dim((1-z)KG)+1=d(C)+d(C^{\perp}).$$ Since $C$ is an MDS code its dual is also an MDS code and therefore, we have $$n+2=d(C)+d(C^{\perp}).$$ 

Now,
\begin{eqnarray*}
 d(C)+d(C^\perp)&\leq & |\mathrm{support~of~} z| + |\mathrm{support ~of~}(1-z)|, \\
                            &\leq & 2 \cdot |\mathrm{support ~of~}z| + 1,~\because~|\mathrm{support~ of~}(1-z)|\leq |\mathrm{support~ of~}z|+1,\\
                            &\leq & 2\cdot |H|+1;
\end{eqnarray*}
implies $|G|<2\cdot |H|.$ And hence $G=H$.
\end{proof}

\textbf{Remark} In \cite{CruzWolfgang}, the Lemma $4.4$  proves ``For $G$ be an abelian group and $K$ be a field of characteristic $p.$ Suppose $C$ is an LCD group code. If $C$ is an MDS code then characteristic of $K$ does not divide $n$, the cardinality of $G$.'' But they have not concluded anything when $G$ is non abelian. The question we will address here has been asked in \cite{CruzWolfgang}, it says for $G$ being an arbitrary finite group and $KG$ contains an LCD, MDS group code, then is it always true that the characteristic of $K$ does not divide $|G|$? Although we have not produced a definite answer to this question, but we have produced a result which answers the question in affirmation, not for arbitrary groups but for a class of groups which is quite larger than the class of finite commutative groups. And the Lemma 4.4 proved in \cite{CruzWolfgang} comes as a corollary to our result.

\begin{theorem}
For $G$ is a finite group such that each Sylow subgroup of $G$ is unique, and let $K$ be a field of characteristic $p$. If $KG$ contains an LCD group code $C$ which is also an MDS code, then $p\nmid |G|$.
\end{theorem}

\begin{proof}
Suppose $p=p_1\mid |G|$ and $|G|=p_1^{\alpha_1}\times p_2^{\alpha_2}\times \cdots \times p_k^{\alpha_k}$,  where $p_1,p_2,\cdots ,p_k$ are distinct primes. Suppose $P_i$ denotes the Sylow-$p_i$-subgroup of $G$, ($1\leq i\leq k$). Now $C$ is an LCD group code, so $C$ is of the form $zKG$, where $z=z^2=\hat z$. Since $z^{p_1^{\alpha_1}}=z$, so components of $z\subseteq \lbrace $\{Sylow-$p_1$-subgroup\}$\setminus\lbrace id\rbrace  \rbrace^c$, complement of $\lbrace$\{Sylow-$p_1$-subgroup\}$\setminus\lbrace id\rbrace  \rbrace$. Again for $C$ being an LCD and MDS code, components of $z$ generates $G$, therefore for $x\in P_1$, $x=a_1\ast a_2\ast \cdots \ast a_t$, where $a_i$ comes from the components of $z$, and may not be distinct. Now for $P_i\neq P_j$ are two Sylow subgroups of $|G|$ other than $P_1$, then $P_iP_j$ is also a subgroup of $G$, since they are normal in $G$, and $|P_iP_j|=\frac{|P_i||P_j|}{|P_i\cap P_j|}$. Since $P_i\cap P_j=\lbrace id\rbrace$, $|P_iP_j|=|P_i||P_j|$. Furthermore $gcd(|P_iP_j|,|P_1|)=1$. In the similar way we can show that for $P_{i_1},P_{i_2},\dots ,P_{i_r}$, being a collection of distinct Sylow subgroups of $G$, other than $P_1$; $P_{i_1}P_{i_2}\cdots P_{i_r}$ forms a subgroup of $G$ and has order co-prime to $p_1$. As $a_1, a_2, \dots, a_t$; come from different Sylow subgroups of $G$ other than $P_1$ so, $a_1\ast a_2\ast \cdots \ast a_t$ has order prime to $p_1$ and hence cannot express a member of $P_1$, therefore we arrive at a contradiction to our hypothesis that $C$ is an LCD and MDS code. So, $p_1\nmid |G|$ and we are done.
\end{proof}

\begin{coro}
For $G$ is a finite abelian group and let $K$ be a field of characteristic $p$. If $KG$ contains an LCD group code $C$ which is also an MDS code, then $p\nmid |G|$.
\end{coro}

\begin{proof}
Since $G$ is finite abelian, all its Sylow subgroups are normal, hence unique. Therefore conclusion drawn from the preceding theorem. 
\end{proof}

\textbf{Alert!!!} There are non-commutative groups, having unique Sylow subgroups. Let $U$ be the group of all upper triangular matrices, having $1$ in diagonal and entries from $Z_3$. Then $|U|=27$, and consider $G=D_4\times U$, where $D_4$ is the dihedral group of order $8$. Then $G$ is non abelian and $|G|=2^3\times 3^3$. The lone Sylow-2-subgroup and Sylow-3-subgroup of $G$ are $D_4$ and $U$ respectively.

\begin{theorem}
For $G$ is a finite group and $K$ is a field of characteristic $p$ and $p\mid |G|$. Theorem $10$ can be made stronger, without the requirement that Sylow-$p$-subgroup is unique.
\end{theorem}

\begin{proof}
The proof is exactly same as Theorem $10$.
\end{proof}

\begin{coro}
For $G$ is a finite group and $K$ is a field of characteristic $p$ ($p$ is an odd prime) and $p\mid |G|.$ If $|G|=2\cdot p^n$, then $KG$ cannot contain an LCD and MDS group code.
\end{coro}

\begin{proof}
If $KG$ contains an LCD and MDS group code $C$, $C$ is of the form $zKG$, where $z=z^2=\hat z$, and components of $z$ generate $G$, but all the components of $z$ has order $2$, and therefore group generated by them is abelian always, then by Corollary $7$, we get a contradiction.
\end{proof}

\begin{coro}
For $K$ is a field of characteristic $p$ and $G$ is finite group with $|G|=p\cdot q^m$ ($p$ and $q$ are distinct primes, $q>p)$, then $KG$ cannot contain an LCD, MDS group code.
\end{coro}

\section{Examples}
In this section we are giving some examples in context of our theorems.
\begin{ex}
Let $G=S_3$ and $K$ be a field of characteristic $2$, since $S_3$ contains a normal subgroup having order $2k+1,(k=1)$, namely $A_3$, so it is possible to construct a self-adjoint LCD group code. We put $z=1.e+1.(123)+1.(132)$, then $z=z^2=\hat z$, and therefore $zKS_3$ forms a self-adjoint LCD group code.
\end{ex}

\begin{ex}
Let $zKG$ is a self-adjoint group code, components of $z$ generate a normal subgroup, but for every normal subgroup we may not be able to construct a self-adjoint group code by the method we have described in Theorem $3$, for example if $K=\mathbb{Z}_2$ and $G=\mathbb{Z}_8$, then $G$ contains a subgroup $\lbrace \bar0,\bar 2,\bar 4,\bar 6\rbrace$, which is isomorphic to $\mathbb{Z}_4$. Since it idex in $\mathbb{Z}_8$ is $2$ so it is normal in $\mathbb{Z}_8$.  But if we put $z=\bar 0 + \bar 2 + \bar 4 + \bar 6$, then $z$ is self adjoint but not an idempotent as $z^2=0$.
\end{ex}

However, if we have had a normal subgroup of order $pk+1$, where $p$ being the characteristic of the field $k\in N$, then we can always use the method of Theorem $3$ to construct a self adjoint group code.

\begin{ex}
Consider $G=A_4$ and $K$ be any field of characteristic $3$. Now it can be seen very easily that $K_4\triangleleft A_4$, where $K_4=\lbrace (),(12)(34),(13)(24),(14)(23)\rbrace$ is the Klien's four group. Since $K_4$ has order of the form $3k+1,(k=1)$, we can form a self adjoint LCD code by setting $z= () + (12)(34) + (13)(24)+(14)(23)$, we see that $z$ is idempotent, self adjoint, and its components form a normal subgroup in $A_4$.
\end{ex}
For MDS and LCD group code $zKG$, components of $z$ generate $G$, but the converse does not hold. The example as follows.
\begin{ex}
Consider $G=A_4$, it is a well known fact that $A_4$ generated by all three cycles. Consider $K=Z_2$, and $z$ to be the sum of all three cycles, then $z=z^2=\hat z$, then $zKG$ is an LCD group code, where components of $z$ generate $G$, but from direct verification, we see that the dimension of $zKG$ is $3$ but minimum distance $d \neq 10.$
\end{ex}

\begin{ex}
Let $G$ be group of order $75$, and $K$ is any field of characteristic $3$. We should note $5\geq 3$ so $5k+1\mid 75$ iff $k=0$, implies Sylow-5-subgroup of $G$ is unique and hence $KG$ can not contain an LCD, MDS group code no matter $G$ is abelian or not.
\end{ex}

\textbf{Remark} If $G=$ (Sylow-$5$-subgroup) $\rtimes$ (Sylow-$3$-subgroup). $G$ is not abelian.

\begin{ex}
Let $G$ be group of order $2079$, and $K$ is any field of characteristic $3$. Now $2079=3^3 \times 7 \times 11$. Here Sylow-$7$-subgroup and Sylow-$11$-subgroup are unique. Hence $KG$ can not contain an LCD, MDS group code no matter $G$ is abelian or not.
\end{ex}

\textbf{Remark} If we choose Sylow-$3$-subgroup of $G$ as $(\mathbb{Z}_3 \times \mathbb{Z}_3) \rtimes \mathbb{Z}_3$, then $G$ is non abelian.

\section{Conclusion} In this paper, we have studied LCD group code and gave an estimation of the distance of the code, also we explored the involvement of normal subgroups of $G$ with a given LCD self-adjoint  group code and further we proved the Lemma $4.4$ in \cite{CruzWolfgang} for a larger class of finite groups. We cordially invite the readers to upgrade or counter the Theorem $11$  for more larger class of finite groups.\\

\textbf{Acknowledgement}
 
The author Saikat Roy is thankful to CSIR, MHRD, India; for financial support to pursue his research work and the author Satya Bagchi thanks to DST-SERB project grant EEQ/2016/000140; for financial support.




\end{document}